\documentclass[12pt,reqno]{article}

\usepackage[usenames]{color}
\usepackage{amssymb}
\usepackage{amsmath}
\usepackage{amsthm}
\usepackage{amsfonts}
\usepackage{amscd}
\usepackage{graphicx}

\usepackage[colorlinks=true,
linkcolor=webgreen,
filecolor=webbrown,
citecolor=webgreen]{hyperref}

\definecolor{webgreen}{rgb}{0,.5,0}
\definecolor{webbrown}{rgb}{.6,0,0}

\usepackage{color}
\usepackage{fullpage}
\usepackage{float}

\usepackage{graphics}
\usepackage{latexsym}
\usepackage{epsf}
\usepackage{accents}

\newcommand{\seqnum}[1]{\href{https://oeis.org/#1}{\rm \underline{#1}}}

\begin{document}

\theoremstyle{plain}
\newtheorem{theorem}{Theorem}
\newtheorem{corollary}[theorem]{Corollary}
\newtheorem{lemma}{Lemma}
\newtheorem{example}{Example}

\begin{center}
\vskip 1cm{\LARGE\bf 
Binomial Fibonacci Power Sums\\
}
\vskip 1cm
\large
Kunle Adegoke \\
Department of Physics and Engineering Physics\\
Obafemi Awolowo University\\
220005 Ile-Ife, Nigeria\\
\href{mailto:adegoke00@gmail.com}{\tt adegoke00@gmail.com} \\
\end{center}

\vskip .2 in



\begin{abstract}
\noindent We evaluate various binomial sums involving the powers of Fibonacci and Lucas \mbox{numbers}.
\end{abstract}

\section{Introduction}
Our main goal in this paper is to evaluate the following sums of powers of Fibonacci and Lucas numbers involving the binomial coefficients:
\[
\sum_{k = 0}^n {( \pm 1)^k \binom nkF_{j(rk + s)}^{2m} } ,\quad\sum_{k = 0}^n {( \pm 1)^k \binom nkL_{j(rk + s)}^{2m} }, 
\]

\[
\sum_{k = 0}^n {( \pm 1)^k \binom nkF_{j(2rk + s)}^{2m + 1} } ,\quad\sum_{k = 0}^n {( \pm 1)^k \binom nkL_{j(2rk + s)}^{2m + 1} }; 
\]
thereby extending the work of Wessner~\cite{wessner66}, Hoggatt and Bicknell~\cite{hoggatt64a,hoggatt64b}, Long~\cite{long90}, Kili\c{c} et al~\cite{kilic15} and several previous researchers. Here $n$ is any non-negative integer, $j$, $m$, $r$ and $s$ are any integers and $F_t$ and $L_t$ are the Fibonacci and Lucas numbers.

There is a dearth of binomial cubic Fibonacci and Lucas identities in existing literature. We will show that, for any non-negative integer $n$ and any integer $s$,
\[
\sum_{k = 0}^n {\binom nkF_{k + s}^3 }  = \frac{1}{5}(2^n F_{2n + 3s}  + 3F_{n - s} ),
\]

\[
\sum_{k = 0}^n {\binom nkL_{k + s}^3 }  = 2^n L_{2n + 3s}  + 3L_{n - s},
\]

\[
\sum_{k = 0}^n {\binom nk(-1)^kF_{k + s}^3 }  = \frac{1}{5}((-1)^n2^n F_{n + 3s}  - (-1)^s3F_{2n + s} ),
\]

\[
\sum_{k = 0}^n {(-1)^k\binom nkL_{k + s}^3 }  = (-1)^n2^n L_{n + 3s}  + (-1)^s3L_{2n + s},
\]

\[
\sum_{k = 0}^n {\binom nk2^k F_{k + s}^3 }  = \left\{ \begin{array}{l}
 5^{n/2 - 1} (F_{3n + 3s}  - ( - 1)^s 3F_s ),\quad\mbox{$n$ even}; \\ 
 5^{(n - 3)/2} (L_{3n + 3s}  + ( - 1)^s 3L_s )\quad\mbox{$n$ odd}, \\ 
 \end{array} \right.
\]

\[
\sum_{k = 0}^n {\binom nk2^k L_{k + s}^3 }  = \left\{ \begin{array}{l}
 5^{n/2} (L_{3n + 3s}  + ( - 1)^s 3L_s ),\quad\mbox{$n$ even}; \\ 
 5^{(n + 1)/2} (F_{3n + 3s}  - ( - 1)^s 3F_s )\quad\mbox{$n$ odd}. \\ 
 \end{array} \right.
\]
We will also derive the following binomial summation identities which we believe are new:
\[
\sum_{k = 0}^n {( - 1)^k \binom nkF_{2jr + p}^{n - k} F_p^k F_{j(rk + s)}^2 }  = \frac15(F_{2jr}^n L_{pn - 2js}  - ( - 1)^{js} 2F_{jr}^n L_{jr + p}^n), 
\]

\[
\sum_{k = 0}^n {( - 1)^k \binom nkF_{2jr + p}^{n - k} F_p^k L_{j(rk + s)}^2 }  = F_{2jr}^n L_{pn - 2js}  + ( - 1)^{js} 2F_{jr}^n L_{jr + p}^n, 
\]

\[
\sum_{k = 0}^n {( - 1)^k \binom nkL_{2jr + p}^{n - k} L_p^k F_{j(rk + s)}^2 }  = \left\{ \begin{array}{l}
 5^{n/2 - 1} F_{2jr}^n L_{pn - 2js}  - ( - 1)^{js} 5^{n - 1} 2F_{jr}^n F_{jr + p}^n,\quad\mbox{$n$ even};  \\ 
 5^{(n - 1)/2} F_{2jr}^n F_{pn - 2js}  - ( - 1)^{js} 5^{n - 1} 2F_{jr}^n F_{jr + p}^n,\quad\mbox{$n$ odd},  \\ 
 \end{array} \right.
\]
and
\[
\sum_{k = 0}^n {( - 1)^k \binom nkL_{2jr + p}^{n - k} L_p^k L_{j(rk + s)}^2 }  = \left\{ \begin{array}{l}
 5^{n/2} F_{2jr}^n L_{pn - 2js}  + ( - 1)^{js} 5^{n} 2F_{jr}^n F_{jr + p}^n,\quad\mbox{$n$ even};  \\ 
 5^{(n + 1)/2} F_{2jr}^n F_{pn - 2js}  + ( - 1)^{js} 5^{n} 2F_{jr}^n F_{jr + p}^n,\quad\mbox{$n$ odd}.  \\ 
 \end{array} \right.
\]
The Fibonacci numbers, $F_n$, and the Lucas numbers, $L_n$, are defined, for \text{$n\in\mathbb Z$}, through the recurrence relations 
\begin{equation}\label{eq.s6z1bcx}
F_n=F_{n-1}+F_{n-2}, \text{($n\ge 2$)},\quad\text{$F_0=0$, $F_1=1$};
\end{equation}
and
\begin{equation}
L_n=L_{n-1}+L_{n-2}, \text{($n\ge 2$)},\quad\text{$L_0=2$, $L_1=1$};
\end{equation}
with
\begin{equation}
F_{-n}=(-1)^{n-1}F_n,\quad L_{-n}=(-1)^nL_n.
\end{equation}
Throughout this paper, we denote the golden ratio, $(1+\sqrt 5)/2$, by $\alpha$ and write $\beta=(1-\sqrt 5)/2=-1/\alpha$, so that $\alpha\beta=-1$ and $\alpha+\beta=1$. 

Explicit formulas (Binet formulas) for the Fibonacci and Lucas numbers are
\begin{equation}
F_n  = \frac{{\alpha ^n  - \beta ^n }}{{\alpha  - \beta }},\quad L_n  = \alpha ^n  + \beta ^n,\quad n\in\mathbb Z.
\end{equation}
Koshy \cite{koshy} and Vajda \cite{vajda} have written excellent books dealing with Fibonacci and Lucas numbers.

Our results emanate from the following general Fibonacci and Lucas summation identities (Lemma~\ref{lem.o33oai3}):
\[\tag{BF}
\sum_{k = 0}^n {\binom nkx^{n - k} z^k F_{j(rk + s)}^m }= \frac{1}{{(\sqrt 5 )^m }}\sum_{i = 0}^m {( - 1)^{i(js + 1)} \binom mi\alpha ^{(m - 2i)js} \left( {x + ( - 1)^{ijr} \alpha ^{(m - 2i)jr} z} \right)^n }, 
\]

\[\tag{BL}
\sum_{k = 0}^n {\binom nkx^{n - k} z^k L_{j(rk + s)}^m }= \sum_{i = 0}^m {( - 1)^{ijs} \binom mi\alpha ^{(m - 2i)js} \left( {x + ( - 1)^{ijr} \alpha ^{(m - 2i)jr} z} \right)^n }.
\]
For low $m$, the identities (BF) and (BL) are more useful in the equivalent form
\[\tag{BF$^{'}$}
\sum_{k = 0}^n {\binom nkx^{n - k} z^k F_{j(rk + s)}^m }  = \frac{1}{{(\sqrt 5 )^m }}\sum_{i = 0}^m {( - 1)^i \binom mi\beta ^{ijs} \alpha ^{(m - i)js} \left( {x + \beta ^{ijr} \alpha ^{(m - i)jr} z} \right)^n },
\]

\[\tag{BL$^{'}$}
\sum_{k = 0}^n {\binom nkx^{n - k} z^k L_{j(rk + s)}^m }  = \sum_{i = 0}^m {\binom mi\beta ^{ijs} \alpha ^{(m - i)js} \left( {x + \beta ^{ijr} \alpha ^{(m - i)jr} z} \right)^n }.
\]
When $m=1$, we have the weighted linear binomial identities:
\[\tag{F1}
\sum_{k = 0}^n {\binom nkx^{n - k} z^k F_{j(rk + s)} }  = \frac{1}{{\sqrt 5 }}\left( {\alpha ^{js} (x + \alpha ^{jr} z)^n  - \beta ^{js} (x + \beta ^{jr} z)^n } \right),
\]

\[\tag{L1}
\sum_{k = 0}^n {\binom nkx^{n - k} z^k L_{j(rk + s)} }  = \alpha ^{js} (x + \alpha ^{jr} z)^n  + \beta ^{js} (x + \beta ^{jr} z)^n,
\]
which are valid for $n$ a non-negative integer, $j$, $r$, $s$ integers and real or complex $x$ and $z$.
Most linear binomial Fibonacci identities can be obtained from identities (F1) and (L1) by substituting appropriate choices of $n$, $j$, $r$, $s$, $x$ and $z$.
For example, if we write $2r$ for $r$ and set $x=(-1)^{jr}$, $z=1$ in (F1) and (L1), we obtain
\begin{equation}\label{eq.dg8fc1k}
\sum_{k = 0}^n {( - 1)^{jrk} \binom nkF_{j(2rk + s)} }  = ( - 1)^{jrn} L_{jr}^n F_{j(rn + s)},
\end{equation}

\begin{equation}\label{eq.qnc2vt5}
\sum_{k = 0}^n {( - 1)^{jrk} \binom nkL_{j(2rk + s)} }  = ( - 1)^{jrn} L_{jr}^n L_{j(rn + s)};
\end{equation}
which are valid for $n$ a non-negative integer and integers $r$, $s$ and $j$.
The special case ($s=0$) of identity~\eqref{eq.qnc2vt5} was also derived by Layman~\cite{layman77}.
As another example of linear binomial Fibonacci identities that may be derived from (F1) and (L1),
write $2r$ for $r$ and set $x=(-1)^{jr}$, $z=-1$. This gives
\begin{equation}
\sum_{k = 0}^n {( - 1)^{(jr + 1)k} \binom nkF_{j(2rk + s)} }  = \left\{ \begin{array}{l}
 5^{n/2} F_{jr}^n F_{j(rn + s)},\quad\mbox{$n$ even;}  \\
 \\ 
 ( - 1)^{jr + 1} 5^{(n - 1)/2} F_{jr}^n L_{j(rn + s)},\quad\mbox{$n$ odd};  \\ 
 \end{array} \right.
\end{equation}

\begin{equation}
\sum_{k = 0}^n {( - 1)^{(jr + 1)k} \binom nkL_{j(2rk + s)} }  = \left\{ \begin{array}{l}
 5^{n/2} F_{jr}^n L_{j(rn + s)},\quad\mbox{$n$ even;}  \\
 \\ 
 ( - 1)^{jr + 1} 5^{(n + 1)/2} F_{jr}^n F_{j(rn + s)},\quad\mbox{$n$ odd}.  \\ 
 \end{array} \right.
\end{equation}
Setting ($x=F_{p + jr}$, $z=-F_p$) and also ($x=L_{p + jr}$, $z=-L_p$) and making use of the identities of Hoggat et al, (see Lemma~\ref{lem.ydalnfx}), where $p$ is any integer, we find
\begin{equation}\label{eq.yyj9a81}
\sum_{k = 0}^n {( - 1)^k \binom nkF_{p + jr}^{n - k} F_p^k F_{j(rk + s)} }  = ( - 1)^{js + 1} F_{jr}^n F_{pn - js};
\end{equation}

\begin{equation}\label{eq.hvteivf}
\sum_{k = 0}^n {( - 1)^k \binom nkF_{p + jr}^{n - k} F_p^k L_{j(rk + s)} }  = ( - 1)^{js + 1} F_{jr}^n L_{pn - js};
\end{equation}
and
\begin{equation}\label{eq.h2mmlc6}
\sum_{k = 0}^n {( - 1)^k \binom nkL_{p + jr}^{n - k} L_p^k F_{j(rk + s)} }  = \left\{ \begin{array}{l}
 ( - 1)^{js + 1} 5^{\frac{n}{2}} F_{jr}^n F_{pn - js},\quad\mbox{$n$ even};  \\ 
  \\ 
 ( - 1)^{js + 1} 5^{\frac{{n - 1}}{2}} F_{jr}^n L_{pn - js},\quad\mbox{$n$ odd};  \\ 
 \end{array} \right.
\end{equation}

\begin{equation}\label{eq.r52qdzt}
\sum_{k = 0}^n {( - 1)^k \binom nkL_{p + jr}^{n - k} L_p^k L_{j(rk + s)} }  = \left\{ \begin{array}{l}
 ( - 1)^{js} 5^{\frac{n}{2}} F_{jr}^n L_{pn - js},\quad\mbox{$n$ even};  \\ 
  \\ 
 ( - 1)^{js} 5^{\frac{{n + 1}}{2}} F_{jr}^n F_{pn - js},\quad\mbox{$n$ odd}.  \\ 
 \end{array} \right.
\end{equation}
Identities~\eqref{eq.yyj9a81}, \eqref{eq.hvteivf} were obtained by Carlitz~\cite{carlitz78} while these and~\eqref{eq.h2mmlc6} and~\eqref{eq.r52qdzt} may be found in~Dresel~\cite{dresel93}. The special case ($s=0$) of \eqref{eq.yyj9a81} was also derived by Layman~\cite{layman77}.
\section{Required identities and preliminary results}
\begin{lemma}
For real or complex $z$, let a given well-behaved function $h(z)$ have, in its domain, the representation $h(z)=\sum_{k=c_1}^{c_2}{g(k)z^{f(k)}}$ where $f(k)$ and $g(k)$ are given real sequences and \mbox{$c_1, c_2\in [-\infty,\infty]$}. Let $j$ be an integer. Then,
\[\tag{F}
\sum_{k = c_1}^{c_2} {g(k)z^{f(k)} F_{jf(k)}^m }  = \frac{1}{{(\sqrt 5 )^m }}\sum_{i = 0}^m {( - 1)^i \binom mih\left( {\beta ^{ij} \alpha ^{(m - i)j} z} \right)}, 
\]

\[\tag{L}
\sum_{k = c_1}^{c_2} {g(k)z^{f(k)} L_{jf(k)}^m }  = \sum_{i = 0}^m {\binom mih\left( {\beta ^{ij} \alpha ^{(m - i)j} z} \right)}. 
\]
\end{lemma}
\begin{proof}
We have
\[
\begin{split}
\sum_{k = c_1 }^{c_2 } {g(k)z^{f(k)} F_{jf(k)}^m }  &= \sum_{k = c_1 }^{c_2 } {g(k)z^{f(k)} \frac{{\left( {\alpha ^{jf(k)}  - \beta ^{jf(k)} } \right)^m }}{{(\sqrt 5 )^m }}}\\
&= \frac{1}{{(\sqrt 5 )^m }}\sum_{k = c_1 }^{c_2 } {g(k)z^{f(k)} \sum_{i = 0}^m {( - 1)^i \binom mi\beta ^{ijf(k)} \alpha ^{(m - i)jf(k)} } }\\ 
&= \frac{1}{{(\sqrt 5 )^m }}\sum_{i = 0}^m {( - 1)^i \binom mi\sum_{k = c_1 }^{c_2 } {g(k)\left( {\beta ^{ij} \alpha ^{(m - i)j} z} \right)^{f(k)} } }\\ 
&= \frac{1}{{(\sqrt 5 )^m }}\sum_{i = 0}^m {( - 1)^i \binom mih\left( {\beta ^{ij} \alpha ^{(m - i)j} z} \right)}. 
\end{split}
\]
The proof of {\rm (L)} is similar.
\end{proof}
Since $\beta^i\alpha^{m - i} = (-1)^i\alpha^{m - 2i}$, identities {\rm (F)} and {\rm (L)} can also be written as
\[\tag{F$^{'}$}
\sum_{k = c_1}^{c_2} {g(k)z^{f(k)} F_{jf(k)}^m }  = \frac{1}{{(\sqrt 5 )^m }}\sum_{i = 0}^m {( - 1)^i \binom mih\left( {(-1) ^{ij} \alpha ^{(m - 2i)j} z} \right)}, 
\]

\[\tag{L$^{'}$}
\sum_{k = c_1}^{c_2} {g(k)z^{f(k)} L_{jf(k)}^m }  = \sum_{i = 0}^m {\binom mih\left( {(-1) ^{ij} \alpha ^{(m - 2i)j} z} \right)}. 
\]
\begin{lemma}\label{lem.o33oai3}
For non-negative integers $m$ and $n$, integers $j$, $r$ and $s$ and real or complex $x$ and $z$,
\[\tag{BF}
\sum_{k = 0}^n {\binom nkx^{n - k} z^k F_{j(rk + s)}^m }= \frac{1}{{(\sqrt 5 )^m }}\sum_{i = 0}^m {( - 1)^{i(js + 1)} \binom mi\alpha ^{(m - 2i)js} \left( {x + ( - 1)^{ijr} \alpha ^{(m - 2i)jr} z} \right)^n }, 
\]

\[\tag{BL}
\sum_{k = 0}^n {\binom nkx^{n - k} z^k L_{j(rk + s)}^m }=\sum_{i = 0}^m {( - 1)^{ijs} \binom mi\alpha ^{(m - 2i)js} \left( {x + ( - 1)^{ijr} \alpha ^{(m - 2i)jr} z} \right)^n }.
\]
\end{lemma}
\begin{proof}
Consider the binomial identity
\begin{equation}
h(z) = \sum_{k = 0}^n {g(k) z^{f(k)} }  = z^s (x + z^r )^n,
\end{equation}
where 
\begin{equation}\label{eq.rm3s06v}
f(k)=rk+s,\quad g(k)=\binom nkx^{n-k}.
\end{equation}
Thus,
\begin{equation}\label{eq.mcl68vg}
h\left( {(-1) ^{ij} \alpha ^{(m - 2i)j} z} \right)=(-1) ^{ijs} \alpha ^{(m - 2i)js} z^s(x + (-1) ^{ijr} \alpha ^{(m - 2i)jr} z^r)^n.
\end{equation}
Use of~\eqref{eq.rm3s06v} and~\eqref{eq.mcl68vg} in identity~({\rm F}$^{'}$), with $c_1=0$, $c_2=n$ gives
\[
\begin{split}
\sum_{k = 0}^n {\binom nkx^{n - k} z^{rk} F_{j(rk + s)}^m }= \frac{1}{{(\sqrt 5 )^m }}\sum_{i = 0}^m {( - 1)^{i(js + 1)} \binom mi\alpha ^{(m - 2i)js} \left( {x + ( - 1)^{ijr} \alpha ^{(m - 2i)jr} z^r} \right)^n }, 
\end{split}
\]
from which identity {(\rm BF)} follows when we write $z^{1/r}$ for $z$.
To prove {\rm (BL)}, use~\eqref{eq.rm3s06v} and~\eqref{eq.mcl68vg} in identity~({\rm L}$^{'}$).
\end{proof}
It is sometimes convenient to use the (\mbox{$\alpha$ vs $\beta$})  version of identities {(\rm BF)} and {(\rm BL)}:
\[\tag{BF$^{'}$}
\sum_{k = 0}^n {\binom nkx^{n - k} z^k F_{j(rk + s)}^m }  = \frac{1}{{(\sqrt 5 )^m }}\sum_{i = 0}^m {( - 1)^i \binom mi\beta ^{ijs} \alpha ^{(m - i)js} \left( {x + \beta ^{ijr} \alpha ^{(m - i)jr} z} \right)^n },
\]

\[\tag{BL$^{'}$}
\sum_{k = 0}^n {\binom nkx^{n - k} z^k L_{j(rk + s)}^m }  = \sum_{i = 0}^m {\binom mi\beta ^{ijs} \alpha ^{(m - i)js} \left( {x + \beta ^{ijr} \alpha ^{(m - i)jr} z} \right)^n }.
\]
\begin{lemma}[Hoggatt et al~\cite{hoggatt71}]\label{lem.ydalnfx}
For $p$ and $q$ integers,
\begin{gather}
L_{p + q}  - L_p \alpha ^q  =  - \beta ^p F_q \sqrt 5,\\ 
L_{p + q}  - L_p \beta ^q  = \alpha ^p F_q \sqrt 5,\\ 
F_{p + q}  - F_p \alpha ^q  = \beta ^p F_q,\\ 
F_{p + q}  - F_p \beta ^q  = \alpha ^p F_q. 
\end{gather}
\end{lemma}
Quadratic binomial Fibonacci identities may be obtained from $m=2$ in (BF$^{'}$) and (BL$^{'}$):
\[\tag{F2}
\begin{split}
5\sum_{k = 0}^n {\binom nkx^{n - k} z^k F_{j(rk + s)}^2 }  &= \alpha ^{2js} (x + \alpha ^{2jr} z)^n  + \beta ^{2js} (x + \beta ^{2jr} z)^n\\ 
&\qquad - 2( - 1)^{js} (x + ( - 1)^{jr} z)^n, 
\end{split}
\]

\[\tag{L2}
\begin{split}
\sum_{k = 0}^n {\binom nkx^{n - k} z^k L_{j(rk + s)}^2 }  &= \alpha ^{2js} (x + \alpha ^{2jr} z)^n  + \beta ^{2js} (x + \beta ^{2jr} z)^n\\ 
&\qquad + 2( - 1)^{js} (x + ( - 1)^{jr} z)^n.
\end{split}
\]
\begin{theorem}
For non-negative integer $n$ and integers $j$, $r$, $s$, $p$,
\begin{equation}\label{eq.ugl2run}
\sum_{k = 0}^n {( - 1)^k \binom nkF_{2jr + p}^{n - k} F_p^k F_{j(rk + s)}^2 }  = \frac15(F_{2jr}^n L_{pn - 2js}  - ( - 1)^{js} 2F_{jr}^n L_{jr + p}^n),\quad\mbox{$p\ne0$}, 
\end{equation}

\begin{equation}\label{eq.zq8hing}
\sum_{k = 0}^n {( - 1)^k \binom nkF_{2jr + p}^{n - k} F_p^k L_{j(rk + s)}^2 }  = F_{2jr}^n L_{pn - 2js}  + ( - 1)^{js} 2F_{jr}^n L_{jr + p}^n,\quad\mbox{$p\ne0$}, 
\end{equation}

\begin{equation}\label{eq.qlndhp6}
\sum_{k = 0}^n {( - 1)^k \binom nkL_{2jr + p}^{n - k} L_p^k F_{j(rk + s)}^2 }  = \left\{ \begin{array}{l}
 5^{n/2 - 1} F_{2jr}^n L_{pn - 2js}  - ( - 1)^{js} 5^{n - 1} 2F_{jr}^n F_{jr + p}^n,\quad\mbox{$n$ even};  \\ 
 5^{(n - 1)/2} F_{2jr}^n F_{pn - 2js}  - ( - 1)^{js} 5^{n - 1} 2F_{jr}^n F_{jr + p}^n,\quad\mbox{$n$ odd},  \\ 
 \end{array} \right.
\end{equation}
and
\begin{equation}\label{eq.om6lsdf}
\sum_{k = 0}^n {( - 1)^k \binom nkL_{2jr + p}^{n - k} L_p^k L_{j(rk + s)}^2 }  = \left\{ \begin{array}{l}
 5^{n/2} F_{2jr}^n L_{pn - 2js}  + ( - 1)^{js} 5^{n} 2F_{jr}^n F_{jr + p}^n,\quad\mbox{$n$ even};  \\ 
 5^{(n + 1)/2} F_{2jr}^n F_{pn - 2js}  + ( - 1)^{js} 5^{n} 2F_{jr}^n F_{jr + p}^n,\quad\mbox{$n$ odd}.  \\ 
 \end{array} \right.
\end{equation}
\end{theorem}
\begin{proof}
Choose $x=F_{2jr+p}$, $z=-F_p$ in (F2), noting Lemma~\ref{lem.ydalnfx}, to obtain
\[
\begin{split}
5\sum_{k = 0}^n {( - 1)^k \binom nkF_{2jr + p}^{n - k} F_p^k F_{j(rk + s)}^2 }  &= F_{2jr}^n (\alpha ^{2js} \beta ^{pn}  + \alpha ^{pn} \beta ^{2js} )\\
&\qquad - 2( - 1)^{js} (F_{2jr + p}  - ( - 1)^{jr} F_p )^n,
\end{split}
\]
from which identity~\eqref{eq.ugl2run} follows. The same $(x,z)$ choice in (L2) produces identity~\eqref{eq.zq8hing}.

Set $x=L_{2jr+p}$, $z=-L_p$ in (F2), utilizing Lemma~\ref{lem.ydalnfx}. This gives
\[
\begin{split}
5\sum_{k = 0}^n {( - 1)^k \binom nkL_{2jr + p}^{n - k} L_p^k F_{j(rk + s)}^2 }  &= F_{2jr}^n (\sqrt 5 )^n (\alpha ^{pn - 2js}  + (-1)^n\beta ^{pn - 2js} )\\
&\qquad - 2( - 1)^{js} (L_{2jr + p}  - ( - 1)^{jr} L_p )^n; 
\end{split}
\]
and hence identity~\eqref{eq.qlndhp6}. The same $(x,z)$ choice in (L2) produces identity~\eqref{eq.om6lsdf}.
\end{proof}
Cubic binomial Fibonacci identities may be obtained from $m=3$ in (BF$^{'}$) and (BL$^{'}$):
\[\tag{F3}
\begin{split}
5\sqrt 5 \sum_{k = 0}^n {\binom nkx^{n - k} z^k F_{j(rk + s)}^3 }  &= \alpha ^{3js} (x + \alpha ^{3jr} z)^n  - \beta ^{3js} (x + \beta ^{3jr} z)^n\\
&\qquad - ( - 1)^{js} 3\alpha ^{js} (x + ( - 1)^{jr} \alpha ^{jr} z)^n\\
&\quad\qquad + ( - 1)^{js} 3\beta ^{js} (x + ( - 1)^{jr} \beta ^{jr} z)^n, 
\end{split}
\]

\[\tag{L3}
\begin{split}
\sum_{k = 0}^n {\binom nkx^{n - k} z^k L_{j(rk + s)}^3 }  &= \alpha ^{3js} (x + \alpha ^{3jr} z)^n  + \beta ^{3js} (x + \beta ^{3jr} z)^n\\
&\qquad + ( - 1)^{js} 3\alpha ^{js} (x + ( - 1)^{jr} \alpha ^{jr} z)^n\\
&\quad\qquad + ( - 1)^{js} 3\beta ^{js} (x + ( - 1)^{jr} \beta ^{jr} z)^n. 
\end{split}
\]
\begin{theorem}
For non-negative integer $n$ and any integer $s$,
\begin{equation}\label{eq.zfnns38}
\sum_{k = 0}^n {\binom nkF_{k + s}^3 }  = \frac{1}{5}(2^n F_{2n + 3s}  + 3F_{n - s} ),
\end{equation}

\begin{equation}\label{eq.ldhv9n5}
\sum_{k = 0}^n {\binom nkL_{k + s}^3 }  = 2^n L_{2n + 3s}  + 3L_{n - s},
\end{equation}

\begin{equation}\label{eq.kr0hhlt}
\sum_{k = 0}^n {\binom nk(-1)^kF_{k + s}^3 }  = \frac{1}{5}((-1)^n2^n F_{n + 3s}  - (-1)^s3F_{2n + s} ),
\end{equation}

\begin{equation}\label{eq.yryk5gj}
\sum_{k = 0}^n {(-1)^k\binom nkL_{k + s}^3 }  = (-1)^n2^n L_{n + 3s}  + (-1)^s3L_{2n + s},
\end{equation}

\begin{equation}\label{eq.vkvxrgg}
\sum_{k = 0}^n {\binom nk2^k F_{k + s}^3 }  = \left\{ \begin{array}{l}
 5^{n/2 - 1} (F_{3n + 3s}  - ( - 1)^s 3F_s ),\quad\mbox{$n$ even}; \\ 
 5^{(n - 3)/2} (L_{3n + 3s}  + ( - 1)^s 3L_s )\quad\mbox{$n$ odd}, \\ 
 \end{array} \right.
\end{equation}

\begin{equation}\label{eq.hl9in1f}
\sum_{k = 0}^n {\binom nk2^k L_{k + s}^3 }  = \left\{ \begin{array}{l}
 5^{n/2} (L_{3n + 3s}  + ( - 1)^s 3L_s ),\quad\mbox{$n$ even}; \\ 
 5^{(n + 1)/2} (F_{3n + 3s}  - ( - 1)^s 3F_s )\quad\mbox{$n$ odd}. \\ 
 \end{array} \right.
\end{equation}
\end{theorem}
\begin{proof}
Set $x=1$, $z=1$, $j=1$, $r=1$ in (F3) to obtain
\[
5\sqrt 5 \sum_{k = 0}^n {\binom nkF_{k + s}^3 }  = 2^n (\alpha ^{3s + 2n}  - \beta ^{3s + 2n} ) + 3(\alpha ^{n - s}  - \beta ^{n - s} );
\]
and hence identity~\eqref{eq.zfnns38}. To prove identity~\eqref{eq.ldhv9n5}, use these $(x,z,j,\ldots)$ values in (L3). To prove identity~\eqref{eq.kr0hhlt}, set $x=1$, $z=-1$, $j=1$, $r=1$ in (F3) to get
\[
5\sqrt 5 \sum_{k = 0}^n {( - 1)^k \binom nkF_{k + s}^3 }  = ( - 1)^n 2^n (\alpha ^{n + 3s}  - \beta ^{n + 3s} ) - 3( - 1)^s (\alpha ^{2n + s}  - \beta ^{2n + s} ),
\]
from which the identity follows. The proof of~\eqref{eq.yryk5gj} is similar. Use these values in (L3). The proof of~\eqref{eq.vkvxrgg} proceeds with the choice $j=1$, $r=1$, $x=1$, $z=2$ in (F3), giving
\[
\begin{split}
5\sqrt 5 \sum_{k = 0}^n {2^k \binom nkF_{k + s}^3 } & = (\sqrt 5 )^n(\alpha ^{3n + 3s}  - ( - 1)^n \beta ^{3n + 3s} )\\
&\qquad- 3( - 1)^{n + s} (\sqrt 5 )^n (\alpha ^s  - ( - 1)^n \beta ^s ),
\end{split}
\]
from which the identity follows in accordance with the parity of $n$.
The proof of~\eqref{eq.hl9in1f} is similar. Use these $(x,z,j,\ldots)$ values in (L3).
\end{proof}
\begin{lemma}\label{lem.rational}
Let $a$, $b$, $c$ and $d$ be rational numbers and $\lambda$ an irrational number. Then,
\[
a + \lambda b = c + \lambda d \iff a = c,\quad b = d\,.
\]
\end{lemma}
\begin{lemma}\label{lem.jyv8ck6}
If $m$ is an integer and $(f(i))$ a real sequence, then,
\begin{equation}\label{eq.hccmg0d}
\sum_{i = 0}^{2m} {f(i)}  = f(m) + \sum_{i = 0}^{m - 1} {(f(i) + f(2m - i))},
\end{equation}

\begin{equation}
\sum_{i = 0}^{2m + 1} {f(i)}  = f(2m + 1) - f(m) + \sum_{i = 0}^m {(f(i) + f(2m - i))}.
\end{equation}
\end{lemma}
In particular, if $f(2m-i)=f(i)$, then,
\begin{equation}
\sum_{i = 0}^{2m} {f(i)}  = f(m) + 2\sum_{i = 0}^{m - 1} {f(i)},
\end{equation}

\begin{equation}
\sum_{i = 0}^{2m + 1} {f(i)}  = f(2m + 1) - f(m) + 2\sum_{i = 0}^m {f(i)}.
\end{equation}
\begin{lemma}\label{lem.dgcin1i}
For $p$ and $q$ integers,
\begin{equation}\label{eq.z1yfyw6}
1 + ( - 1)^p \alpha ^{2q}  = \left\{ \begin{array}{l}
 ( - 1)^p \alpha ^q F_q \sqrt 5,\quad\mbox{$p$ and $q$ have different parity;}  \\ 
 ( - 1)^p \alpha ^q L_q,\quad\mbox{$p$ and $q$ have the same parity.}  \\ 
 \end{array} \right.
\end{equation}

\begin{equation}\label{eq.xcamhlp}
1 - ( - 1)^p \alpha ^{2q}  = \left\{ \begin{array}{l}
 ( - 1)^{p - 1} \alpha ^q L_q,\quad\mbox{$p$ and $q$ have different parity};  \\ 
 ( - 1)^{p - 1} \alpha ^q F_q \sqrt 5,\quad\mbox{$p$ and $q$ have the same parity}.  \\ 
 \end{array} \right.
\end{equation}
\end{lemma}
\begin{proof}
We have
\begin{equation}\label{eq.sfygluu}
\begin{split}
( - 1)^{p + q}  + ( - 1)^p \alpha ^{2q} & = \alpha ^{p + q} \beta ^{p + q}  + \alpha ^{p + 2q} \beta ^p\\ 
 &= \alpha ^{p + q} \beta ^p (\alpha ^q  + \beta ^q )\\
 &= ( - 1)^p \alpha ^q L_q. 
\end{split}
\end{equation}
Similarly,
\begin{equation}\label{eq.b9sqn42}
( - 1)^{p + q}  - ( - 1)^p \alpha ^{2q}  = ( - 1)^{p - 1} \alpha ^q F_q \sqrt 5. 
\end{equation}
Corresponding to~\eqref{eq.sfygluu} and~\eqref{eq.b9sqn42} we have
\begin{equation}\label{eq.hu8kfq5}
( - 1)^{p + q}  + ( - 1)^p \beta ^{2q} = ( - 1)^p \beta ^q L_q 
\end{equation}
and
\begin{equation}\label{eq.ok9lld2}
( - 1)^{p + q}  - ( - 1)^p \beta ^{2q}  = ( - 1)^{p} \beta ^q F_q \sqrt 5. 
\end{equation}
\end{proof}
Identities~\eqref{eq.sfygluu}, \eqref{eq.b9sqn42}, \eqref{eq.hu8kfq5} and \eqref{eq.ok9lld2} imply
\begin{gather}
( - 1)^q  + \alpha ^{2q}  = \alpha ^q L_q, \\
( - 1)^q  - \alpha ^{2q}  =  - \alpha ^q F_q \sqrt 5, \\
( - 1)^q  + \beta ^{2q}  = \beta ^q L_q, \\
( - 1)^q  - \beta ^{2q}  = \beta ^q F_q \sqrt 5.
\end{gather}
\section{Main results}
\begin{theorem}\label{thm.kqnrcog}
Let $m$ and $n$ be non-negative integers and let $j$, $r$ and $s$ be any integers. Then,
\begin{equation}\label{eq.ix0nr87}
\begin{split}
&\sum_{k = 0}^n {\binom nkF_{j(rk + s)}^{2m} }\\
& = \left\{ \begin{array}{l}
 5^{ - m} \sum_{i = 0}^{m - 1} {( - 1)^{i(js + jrn + 1)} \binom{2m}iL_{(m - i)jr}^n L_{(m - i)(jrn + 2js)} }  + ( - 1)^{m(js + 1)} \binom{2m}m5^{ - m} 2^n,\quad\mbox{$jmr$ even};  \\ 
 5^{n/2 - m} \sum_{i = 0}^{m - 1} {( - 1)^{i(s + 1)} \binom{2m}iF_{(m - i)jr}^n L_{(m - i)(jrn + 2js)} },\quad\mbox{$jmr$ odd, $n$ even};  \\ 
 5^{(n + 1)/2 - m} \sum_{i = 0}^{m - 1} {( - 1)^{is} \binom{2m}iF_{(m - i)jr}^n F_{(m - i)(jrn + 2js)} },\quad\mbox{$jmr$ odd, $n$ odd};  \\ 
 \end{array} \right.
\end{split}
\end{equation}
\begin{equation}\label{eq.bu0jaqx}
\begin{split}
&\sum_{k = 0}^n {\binom nkL_{j(rk + s)}^{2m} }\\
& = \left\{ \begin{array}{l}
  \sum_{i = 0}^{m - 1} {( - 1)^{i(js + jrn)} \binom{2m}iL_{(m - i)jr}^n L_{(m - i)(jrn + 2js)} }  + ( - 1)^{mjs} \binom{2m}m 2^n,\quad\mbox{$jmr$ even};  \\ 
 5^{n/2} \sum_{i = 0}^{m - 1} {( - 1)^{is} \binom{2m}iF_{(m - i)jr}^n L_{(m - i)(jrn + 2js)} },\quad\mbox{$jmr$ odd, $n$ even};  \\ 
 5^{(n + 1)/2} \sum_{i = 0}^{m - 1} {( - 1)^{i(s + 1)} \binom{2m}iF_{(m - i)jr}^n F_{(m - i)(jrn + 2js)} },\quad\mbox{$jmr$ odd, $n$ odd}.  \\ 
 \end{array} \right.
\end{split}
\end{equation}
\end{theorem}
\begin{proof}
In (BF) write $2m$ for $m$ and set $x=1$ and $z=1$. This gives
\begin{equation}\label{eq.szp3g8y}
5^m \sum_{k = 0}^n {\binom nkF_{j(rk + s)}^{2m} }  = \sum_{i = 0}^{2m} {( - 1)^{i(js + 1)} \binom {2m}i\alpha ^{(m - i)2js} \left( {1 + ( - 1)^{ijr} \alpha ^{(m - i)2jr} } \right)^n }.
\end{equation}
Now, on account of Lemma~\ref{lem.dgcin1i}, identity~\eqref{eq.z1yfyw6}, we have
\begin{equation}\label{eq.bsdre5p}
1 + ( - 1)^{ijr} \alpha ^{(m - i)2jr}  = \left\{ \begin{array}{l}
 ( - 1)^{ijr} \alpha ^{(m - i)jr} L_{(m - i)jr},\quad\mbox{$jrm$ even};  \\ 
 ( - 1)^i \alpha ^{(m - i)jr} F_{(m - i)jr} \sqrt 5,\quad\mbox{$jrm$ odd}.  \\ 
 \end{array} \right.
\end{equation}
Thus, using~\eqref{eq.bsdre5p} in~\eqref{eq.szp3g8y}, we have
\begin{equation}\label{eq.naslz7k}
5^m \sum_{k = 0}^n {\binom nkF_{j(rk + s)}^{2m} }  = \left\{ \begin{array}{l}
 \sum_{i = 0}^{2m} {( - 1)^{i(js + jrn + 1)} \binom {2m}i\alpha ^{(m - i)(jrn + 2js)} L_{(m - i)jr}^n },\quad\mbox{$jrm$ even};  \\ 
 (\sqrt 5 )^n \sum_{i = 0}^{2m} {( - 1)^{i(js + n + 1)} \binom {2m}i\alpha ^{(m - i)(jrn + 2js)} F_{(m - i)jr}^n },\quad\mbox{$jrm$ odd}.  \\ 
 \end{array} \right.
\end{equation}
Observe that the left side of~\eqref{eq.naslz7k} evaluates to a rational number since it is the finite sum of rational numbers.
Since,
\begin{equation}\label{eq.aoru123}
2\alpha ^{(m - i)(jrn + 2js)}  = L_{^{(m - i)(jrn + 2js)} }  + F_{^{(m - i)(jrn + 2js)} } \sqrt 5, 
\end{equation}
identity~\eqref{eq.ix0nr87} now follows by comparing both sides of identity~\eqref{eq.naslz7k} in each case of $jmr$ even or $jmr$ odd, invoking Lemma~\ref{lem.rational} with $\lambda=\sqrt 5$.
Note the use of Lemma~\ref{lem.jyv8ck6}, identity~\eqref{eq.hccmg0d} to re-write the ($i=0\text{ to }2m$) sum. The proof of identity~\eqref{eq.bu0jaqx} is similar; set $x=1$ and $z=1$ in (BL) and write $2m$ for $m$.
\end{proof}
\begin{theorem}\label{thm.f6sz08q}
Let $m$ and $n$ be non-negative integers and let $j$, $r$ and $s$ be any integers. Then,
\begin{equation}\label{eq.t7h5qpn}
\begin{split}
&\sum_{k = 0}^n {(-1)^k\binom nkF_{j(rk + s)}^{2m} }\\
& = \left\{ \begin{array}{l}
 5^{ - m}(-1)^n \sum_{i = 0}^{m - 1} {( - 1)^{i(s + n + 1)} \binom{2m}iL_{(m - i)jr}^n L_{(m - i)(jrn + 2js)} }  + ( - 1)^{s + 1} \binom{2m}m5^{ - m} 2^n,\quad\mbox{$jmr$ odd};  \\ 
 5^{n/2 - m} \sum_{i = 0}^{m - 1} {( - 1)^{i(js + 1)} \binom{2m}iF_{(m - i)jr}^n L_{(m - i)(jrn + 2js)} },\quad\mbox{$jmr$ even, $n$ even};  \\ 
 -5^{(n + 1)/2 - m} \sum_{i = 0}^{m - 1} {( - 1)^{i(js + jr + 1)} \binom{2m}iF_{(m - i)jr}^n F_{(m - i)(jrn + 2js)} },\quad\mbox{$jmr$ even, $n$ odd};  \\ 
 \end{array} \right.
\end{split}
\end{equation}
\begin{equation}\label{eq.il6nji3}
\begin{split}
&\sum_{k = 0}^n {(-1)^k\binom nkL_{j(rk + s)}^{2m} }\\
& = \left\{ \begin{array}{l}
 (-1)^n \sum_{i = 0}^{m - 1} {( - 1)^{i(s + n)} \binom{2m}iL_{(m - i)jr}^n L_{(m - i)(jrn + 2js)} }  + ( - 1)^{s} \binom{2m}m2^n,\quad\mbox{$jmr$ odd};  \\ 
 5^{n/2} \sum_{i = 0}^{m - 1} {( - 1)^{ijs} \binom{2m}iF_{(m - i)jr}^n L_{(m - i)(jrn + 2js)} },\quad\mbox{$jmr$ even, $n$ even};  \\ 
 -5^{(n + 1)/2} \sum_{i = 0}^{m - 1} {( - 1)^{i(js + jr)} \binom{2m}iF_{(m - i)jr}^n F_{(m - i)(jrn + 2js)} },\quad\mbox{$jmr$ even, $n$ odd}.  \\ 
 \end{array} \right.
\end{split}
\end{equation}
\end{theorem}
\begin{proof}
In (BF) write $2m$ for $m$ and set $x=1$ and $z=-1$. This gives
\begin{equation}\label{eq.nk7xftp}
5^m \sum_{k = 0}^n {(-1)^k\binom nkF_{j(rk + s)}^{2m} }  = \sum_{i = 0}^{2m} {( - 1)^{i(js + 1)} \binom {2m}i\alpha ^{(m - i)2js} \left( {1 - ( - 1)^{ijr} \alpha ^{(m - i)2jr} } \right)^n }.
\end{equation}
Now, on account of Lemma~\ref{lem.dgcin1i}, identity~\eqref{eq.xcamhlp}, we have
\begin{equation}\label{eq.o4zh1c8}
1 - ( - 1)^{ijr} \alpha ^{(m - i)2jr}  = \left\{ \begin{array}{l}
 ( - 1)^{ijr - 1} \alpha ^{(m - i)jr} F_{(m - i)jr}\sqrt 5,\quad\mbox{$jrm$ even};  \\ 
 ( - 1)^{i - 1} \alpha ^{(m - i)jr} L_{(m - i)jr},\quad\mbox{$jrm$ odd}.  \\ 
 \end{array} \right.
\end{equation}
Thus, using~\eqref{eq.o4zh1c8} in~\eqref{eq.nk7xftp}, we have
\begin{equation}\label{eq.fwrjf85}
5^m \sum_{k = 0}^n {(-1)^k\binom nkF_{j(rk + s)}^{2m} }  = \left\{ \begin{array}{l}
 \sum_{i = 0}^{2m} {( - 1)^{in+is+i-n} \binom {2m}i\alpha ^{(m - i)(jrn + 2js)} L_{(m - i)jr}^n },\quad\mbox{$jrm$ odd};  \\ 
 (\sqrt 5 )^n \sum_{i = 0}^{2m} {( - 1)^{ijnr+ijs+i-n} \binom {2m}i\alpha ^{(m - i)(jrn + 2js)} F_{(m - i)jr}^n },\quad\mbox{$jrm$ even}.  \\ 
 \end{array} \right.
\end{equation}
The left side of~\eqref{eq.fwrjf85} evaluates to a rational number since it is the finite sum of rational numbers.
Making use of identity~\eqref{eq.aoru123}, identity~\eqref{eq.t7h5qpn} follows by comparing both sides of identity~\eqref{eq.fwrjf85} in each case of $jmr$ even or $jmr$ odd, invoking Lemma~\ref{lem.rational} with $\lambda=\sqrt 5$.
The proof of identity~\eqref{eq.il6nji3} is similar; put $x=1$ and $z=-1$ in (BL) and write $2m$ for $m$.
\end{proof}
The proofs of Theorems~\ref{thm.lsajs74} and~\ref{thm.yew22sz} are similar to those of Theorems~\ref{thm.kqnrcog} and~\ref{thm.f6sz08q}. We therefore omit the details and indicate only the appropriate choices of $x$, $z$, $m$ and $r$ to be made in identities~(BF) and (BL) in each case.
\begin{theorem}\label{thm.lsajs74}
Let $m$ and $n$ be non-negative integers and let $j$, $r$ and $s$ be any integers. Then,
\begin{equation}
\begin{split}
&\sum_{k = 0}^n {\binom nkF_{j(2rk + s)}^{2m + 1} }\\
&  = \left\{ \begin{array}{l}
 5^{ - m} \sum_{i = 0}^m {( - 1)^{i(js + 1)} \binom {2m + 1}iL_{(2m + 1 - 2i)jr}^n F_{(2m + 1 - 2i)(jrn + js)} },\quad\mbox{$jr$ even};  \\ 
 5^{n/2 - m} \sum_{i = 0}^m {( - 1)^{i(js + 1)} \binom {2m + 1}iF_{(2m + 1 - 2i)jr}^n F_{(2m + 1 - 2i)(jrn + js)} },\quad\mbox{$jr$ odd, $n$ even};  \\ 
 5^{(n - 1)/2 - m} \sum_{i = 0}^m {( - 1)^{i(js + 1)} \binom {2m + 1}iF_{(2m + 1 - 2i)jr}^n L_{(2m + 1 - 2i)(jrn + js)} },\quad\mbox{$jr$ odd, $n$ odd};  \\ 
 \end{array} \right.
\end{split}
\end{equation}

\begin{equation}
\begin{split}
&\sum_{k = 0}^n {\binom nkL_{j(2rk + s)}^{2m + 1} }\\
&  = \left\{ \begin{array}{l}
 \sum_{i = 0}^m {( - 1)^{ijs} \binom {2m + 1}iL_{(2m + 1 - 2i)jr}^n L_{(2m + 1 - 2i)(jrn + js)} },\quad\mbox{$jr$ even};  \\ 
 5^{n/2} \sum_{i = 0}^m {( - 1)^{ijs} \binom {2m + 1}iF_{(2m + 1 - 2i)jr}^n L_{(2m + 1 - 2i)(jrn + js)} },\quad\mbox{$jr$ odd, $n$ even};  \\ 
 5^{(n + 1)/2} \sum_{i = 0}^m {( - 1)^{ijs} \binom {2m + 1}iF_{(2m + 1 - 2i)jr}^n F_{(2m + 1 - 2i)(jrn + js)} },\quad\mbox{$jr$ odd, $n$ odd}.  \\ 
 \end{array} \right.
\end{split}
\end{equation}
\end{theorem}
\begin{proof}
Set $x=1$, $z=1$ and write $2m + 1 $ for $m$ and $2r$ for $r$ in identities (BF) and (BL).
Note that
\[
1 + \alpha ^{(2m + 1 - 2i)2jr}  = \left\{ \begin{array}{l}
 \alpha ^{(2m + 1 - 2i)jr} L_{(2m + 1 - 2i)jr},\quad\mbox{$jr$ even};  \\ 
 \alpha ^{(2m + 1 - 2i)jr} F_{(2m + 1 - 2i)jr} \sqrt 5,\quad\mbox{$jr$ odd}.  \\ 
 \end{array} \right.
\]
\end{proof}
\begin{theorem}\label{thm.yew22sz}
Let $m$ and $n$ be non-negative integers and let $j$, $r$ and $s$ be any integers. Then,
\begin{equation}
\begin{split}
&\sum_{k = 0}^n {(-1)^k\binom nkF_{j(2rk + s)}^{2m + 1} }\\
&  = \left\{ \begin{array}{l}
 (-1)^n5^{ - m} \sum_{i = 0}^m {( - 1)^{i(js + 1)} \binom {2m + 1}iL_{(2m + 1 - 2i)jr}^n F_{(2m + 1 - 2i)(jrn + js)} },\quad\mbox{$jr$ odd};  \\ 
 5^{n/2 - m} \sum_{i = 0}^m {( - 1)^{i(js + 1)} \binom {2m + 1}iF_{(2m + 1 - 2i)jr}^n F_{(2m + 1 - 2i)(jrn + js)} },\quad\mbox{$jr$ even, $n$ even};  \\ 
 -5^{(n - 1)/2 - m} \sum_{i = 0}^m {( - 1)^{i(js + 1)} \binom {2m + 1}iF_{(2m + 1 - 2i)jr}^n L_{(2m + 1 - 2i)(jrn + js)} },\quad\mbox{$jr$ even, $n$ odd};  \\ 
 \end{array} \right.
\end{split}
\end{equation}

\begin{equation}
\begin{split}
&\sum_{k = 0}^n {(-1)^k\binom nkL_{j(2rk + s)}^{2m + 1} }\\
&  = \left\{ \begin{array}{l}
 (-1)^n\sum_{i = 0}^m {( - 1)^{ijs} \binom {2m + 1}iL_{(2m + 1 - 2i)jr}^n L_{(2m + 1 - 2i)(jrn + js)} },\quad\mbox{$jr$ odd};  \\ 
 5^{n/2} \sum_{i = 0}^m {( - 1)^{ijs} \binom {2m + 1}iF_{(2m + 1 - 2i)jr}^n L_{(2m + 1 - 2i)(jrn + js)} },\quad\mbox{$jr$ even, $n$ even};  \\ 
 -5^{(n + 1)/2} \sum_{i = 0}^m {( - 1)^{ijs} \binom {2m + 1}iF_{(2m + 1 - 2i)jr}^n F_{(2m + 1 - 2i)(jrn + js)} },\quad\mbox{$jr$ even, $n$ odd}.  \\ 
 \end{array} \right.
\end{split}
\end{equation}
\end{theorem}
\begin{proof}
Put $x=1$, $z=-1$ and write $2m + 1 $ for $m$ and $2r$ for $r$ in identities (BF) and (BL).
Note that
\[
1 - \alpha ^{(2m + 1 - 2i)2jr}  = \left\{ \begin{array}{l}
 -\alpha ^{(2m + 1 - 2i)jr} F_{(2m + 1 - 2i)jr}\sqrt 5,\quad\mbox{$jr$ even};  \\ 
 -\alpha ^{(2m + 1 - 2i)jr} L_{(2m + 1 - 2i)jr},\quad\mbox{$jr$ odd}.  \\ 
 \end{array} \right.
\]
\end{proof}

\hrule

\bigskip

\bigskip
\noindent Concerned with sequences: \seqnum{A000032}, \seqnum{A000045}.

\bigskip
\hrule
\bigskip

\vspace*{+.1in}
\noindent



\end{document}